\documentclass[a4paper]{amsart}  
\usepackage[english]{babel} 
\usepackage[utf8]{inputenc}

\usepackage{amsmath, amssymb, amsthm} 
\usepackage{enumerate}

\newtheorem{thm}{Theorem}[section]
\newtheorem{lem}[thm]{Lemma}

\newtheorem{cor}[thm]{Corollary}

\newtheorem*{dfn}{Definition}
\newtheorem*{que}{Question}

\newtheorem*{nota}{Notation}

\newcommand{\Z}{\mathbb{Z}}
\newcommand{\Q}{\mathbb{Q}}
\newcommand{\C}{\mathbb{C}}

\newcommand{\N}{\mathbb{N}}
\newcommand{\cod}{\textrm{cod}}

\newcommand{\code}{\mathcal{C}}

\title{Positive curvature and the elliptic genus}
\author{Nicolas Weisskopf}
\thanks{The author was supported by the Swiss National Science Foundation project 200020\_149761.}
\address{Nicolas Weisskopf, D\'epartement de math\'ematiques, Universit\'e de Fribourg, Chemin du Mus\'ee 23, 1700 Fribourg, Switzerland}
\email{nic.weisskopf@gmail.com}

\begin{document}   

\begin{abstract}
We prove some results about the vanishing of the elliptic genus on positively curved Spin manifolds with logarithmic symmetry rank. The proofs are based on the rigidity of the elliptic genus and Kennard's improvement of the Connectedness Lemma for transversely intersecting, totally geodesic submanifolds.
\end{abstract}

\maketitle

\thispagestyle{empty} 

\section{Introduction}

Which manifolds are positively curved? This captivating question has intrigued geometers for more than a century and has been solved only in a few special cases. It turns out to be  difficult to  construct metrics of positive sectional curvature. 

Indeed, a glance at the literature confirms that all presently known, simply-connected, positively curved manifolds of dimension greater than twenty-four are the sphere, the complex projective space and the quaternionic projective space. This lack of examples indicates that positive  sectional curvature has a strong impact on the underlying topology and consequently, one aims to exhibit topological obstructions. In this note, we analyze the interplay between positive curvature and cobordism invariants and provide new results about the vanishing of the elliptic genus on positively curved Spin manifolds.

This approach dates back to the sixties. Recall that a genus in the sense of Hirzebruch is a ring homomorphism from the oriented cobordism ring to some unital algebra. Classical examples are the signature $sign(M)$ and the $\hat{A}$--genus $\hat{A}(M)$. In analytic terms, both the signature and the $\hat{A}$--genus can be seen as the index of a first-order elliptic differential operator. Here, the signature is the index of the square root of the Laplacian, whereas for Spin manifolds the $\hat{A}$--genus equals the index of the Dirac operator. These considerations culminated in the celebrated Atiyah-Singer Index Theorem.

Shortly after the Index Theorem was proven, Lichnerowicz \cite{Lic63} provided  an obstruction to positive scalar curvature on Spin manifolds. By using a Bochner-type formula, he showed that the $\hat{A}$--genus vanishes on Spin manifolds carrying positive scalar curvature. This implies, for example, that the $K3$--surface $V^4$ does not carry a metric of positive scalar curvature, since $V^4$ is Spin and $\hat{A}(V^4) \neq 0$.

In the years following Lichnerowicz's result, a combined effort of geometers and topologists led to the classification of simply-connected manifolds of dimension greater than four admitting a metric of positive scalar curvature -- a milestone in Riemannian Geometry. It was shown by Gromov-Lawson \cite{GL80}, Schoen-Yau \cite{SY79} and Stolz \cite{Sto92} that the $\hat{A}$--genus (and, more precisely, its KO--theoretic refinement, the \mbox{$\alpha$--invariant}) forms the only obstruction to positive scalar curvature on simply-connected Spin manifolds of dimension greater than four.

If one strengthens the curvature assumption to non-negative sectional curvature, then the signature also yields an obstruction. This follows from Gromov's Betti number Theorem. In \cite{Gro81}, Gromov proved that the total Betti number of a non-negatively curved manifold is bounded by a constant only depending on the dimension. Therefore, the signature is also bounded and so both classical genera are deeply related to curvature.

In the eighties, a new type of genus, the so-called elliptic genera, emerged from a discussion between topologists, number theorists and physicists \cite{Lan88}. The term elliptic originates from the fact that the logarithm of the genus corresponds to an elliptic integral. As example, we mention the universal elliptic genus, which can be thought of as the equivariant signature on the free loop space. The universal elliptic genus admits the expansion
$$ \phi_0(M^{4k}) = q^{-k/2} \cdot \hat{A}(M, \bigotimes_{\substack{n \; \textrm{odd} \\ n \geq 1}}^{\infty}\Lambda_{-q^n} T_{\mathbb{C}}M \otimes \bigotimes_{\substack{n \; \textrm{even} \\ n \geq 1}}^{\infty} S_{q^n} T_{\mathbb{C}}M),
$$
where each coefficient is a characteristic number.

In light of these new genera, one might again explore their connection to curvature. Indeed, it is fascinating to observe that this approach is still fruitful and has led to new exciting conjectures (see the Stolz conjecture \cite{Sto96}). In this context, Dessai raised the following question in {{\cite[\textit{Question 20, p.575}]{Des07}}}.

\begin{que}
 Let $(M^n,g)$ be a Spin manifold. Is the elliptic genus $\phi_0(M)$ constant as a power series, if $M$ admits a metric of positive sectional curvature?
\end{que}

In addition, Dessai gathered some evidence in (\cite{Des05}, \cite{Des07}) that favors a positive answer to this question. In particular, he showed that the coefficients of $\phi_0(M)$ vanish linearly with the symmetry rank. This question has many interesting facets. Not only does the vanishing of the coefficients give some information on the cobordism type, but it also paves a way to distinguish better between positive sectional curvature and positive Ricci curvature. As a matter of fact, there  are many Ricci positive Spin manifolds, whose elliptic genus is not constant.

This paper is centered around Dessai's question. The novelty of our approach lies in the fact that we obtain vanishing results under logarithmic symmetry ranks. The proofs are based on the rigidity of the elliptic genus and Kennard's Periodicity Theorem \cite{Ken13}, which is an improvement of Wilking's Connectedness Lemma \cite{Wil03} in the case of transversely intersecting, totally geodesic submanifolds.

From now on, a positively curved manifold stands for a manifold with positive sectional curvature.  We present our first result.

\begin{thm} \label{thmA} Let $(M^n,g)$ be a closed, positively curved Spin manifold. Suppose that a torus $T^s$ acts isometrically and effectively on $M^n$. Then, $M$ is rationally $4$--periodic or the first $\min \{ \left \lfloor \frac{n}{16} \right \rfloor +1, 2^{s-3} \}$ coefficients of $\phi_0(M)$ vanish.
\end{thm}

Note that a simply-connected, rationally $4$--periodic manifold $M^{4k}$ has the  same rational cohomology as a sphere, a complex or quaternionic projective space. We emphasize that the two statements on the cohomology ring and the elliptic genus are not  opposed. Instead, we rather believe that under the given curvature and symmetry assumptions, the elliptic genus of a rationally $4$--periodic manifold is constant. However, it is noteworthy that the cobordism type of a mani\-fold can in general not be detected via  cohomology. If we rephrase the statement  in terms of the symmetry rank, we observe that a logarithmic bound  becomes apparent. 

\begin{cor} \label{corA}
Let $(M^n,g)$ be a closed, positively curved Spin manifold. Suppose that $(M^n,g)$ has symmetry rank
$$ \textrm{symrank}(M^n,g) > \log_2(n)-1.
$$
Then, $M$ is rationally $4$--periodic or the first $ \left \lfloor \frac{n}{16} \right \rfloor +1$ coefficients of $\phi_0(M)$ vanish.
\end{cor}

Theorem \ref{thmA} indicates a certain trade-off between the cohomology and the elliptic genus. In fact, if we keep the same amount of symmetry and wish for a stronger vanishing result on the elliptic genus, then we can not recover the entire cohomology ring. The next result illustrates this aspect.

\begin{thm} \label{thmB} 
Let $(M^{n},g)$ be a closed, positively curved Spin manifold. Suppose that a torus $T^s$ acts isometrically and effectively on $M^n$. Then, one of the following statement holds:

\begin{enumerate}[1.)]

\item There is an element $x \in H^{4}(M; \Q)$ such that $x^{n/4} \in H^{n}(M; \Q)$ is a generator.

\item The first $\min \{ \left \lfloor \frac{n}{12} \right \rfloor +1, 2^{s-3} \}$ coefficients of $\phi_0(M)$ vanish.
\end{enumerate}
\end{thm}

Another interesting variant turns up, when we deal with a positively curved manifold $(M^n,g)$ with $b_4(M)=0$. In high dimensions, the known examples suggest that $M$ should resemble a sphere.  Moreover, a famous result by Smith states that the fixed point set of a smooth circle action on a sphere consists of a cohomology sphere and is therefore either connected or consists of two isolated fixed points. A similar configuration is studied in the next theorem.

\begin{thm} \label{thmC} 
Let $(M^n,g)$ be a closed and positively curved Spin manifold with $b_4(M)=0$. Suppose that $(M^n,g)$ has symmetry rank
$$ \textrm{symrank}(M^n,g) \geq \log_2(n)
$$
and that the torus fixed point set is connected. Then, the elliptic genus is constant.
\end{thm}

Amann and Kennard \cite[\textit{Theorem A}]{AK14a}  showed that under a symmetry rank of at least $\log_{4/3}(n-3)$ the torus fixed point set of a positively curved manifold $(M^n,g)$ with $b_4(M)=0$ is  a rational cohomology sphere. Combining this with our statement, we obtain the following interesting vanishing result.

\begin{cor} \label{corC} 
Let $(M^n,g)$ be a closed and positively curved Spin manifold with $b_4(M)=0$. Suppose that $(M^n,g)$ has symmetry rank
$$ \textrm{symrank}(M^n,g) > \log_{4/3}(n-3).
$$
Then, the elliptic genus vanishes.
\end{cor}

These three theorems reflect the  interplay between positive  curvature and the elliptic genus. For non-negative curvature, possible connections were explored by Herrmann and the author in \cite{HW14}. Before going any further, we sketch the proof strategy. The proof is a delicate combination of topological, geometric and symmetric arguments.  The key element is to characterize the fixed point set. 

On the topological side, we essentially make use of the rigidity property \cite{BT89} of the elliptic genus. This enables us to compute the elliptic genus in terms of the fixed point set of involutions. In fact, a result by Hirzebruch and Slodowy \cite{HS90} ensures the existence of high-dimensional fixed point components, if the elliptic genus does not vanish. This feature allows us to maintain control on the dimension of the fixed point set. On the geometric side, Frankel's Intersection Theorem \cite{Fra61} and Wilking's Connectedness Lemma \cite{Wil03} characterize the position and the topology of the fixed point components. Finally, Kennard's Periodicity Theorem \cite{Ken13} asserts that two fixed point components, which intersect transversely, give a periodicity on the level of cohomology.

As a consequence, if the elliptic genus does not vanish, it remains to find two fixed point components of high dimension, which intersect transversely. At this point, the symmetric properties set in.  Here, we choose the language of error-correcting codes to describe the symmetric  structure of involutions. We believe that this is a convenient way to gain a systematic overview of the various fixed point sets. Eventually, it turns out that a logarithmic symmetry rank guarantees the existence  of transversely intersecting fixed point components.  These considerations are presented in Lemma \ref{4per_fam}.

The paper is structured as follows. The next section summarizes the geometric and topological methods needed for the proofs. In particular, we recall some useful properties of totally geodesic submanifolds and discuss the rigidity of the elliptic genus. In section 3, we develop some elementary coding theory. The proofs of the theorems are given in section 4. 

\subsection*{Acknowledgments} The results in this paper are part of the author's doctoral thesis. It is a great pleasure for the author to thank his advisor Anand Dessai for introducing the subject to him and for many helpful discussions. Moreover, the author is grateful to Lee Kennard for useful comments on  a previous version. Finally, the author is thankful to Micha Wasem for a long mathematical friendship.

\section{Geometric and topological background}

Throughout this paper all manifolds are assumed to be closed, oriented and smooth. Furthermore, all actions are smooth. We begin with a short review on totally geodesic submanifolds and the elliptic genus. 

\subsection{Totally geodesic submanifolds} Let $(M^n,g)$ be a positively curved Riemannian manifold. A submanifold $N \subset M$ is called \textit{totally geodesic}, if any geodesic of $N$ with respect to the induced metric is also a geodesic of $M$. It is well-known that in the presence of symmetry, totally geodesic submanifolds arise naturally as fixed point sets. As a first theorem we  mention Frankel's Intersection Theorem.

\begin{thm}[Intersection Theorem {{\cite[\textit{Theorem 1, p.169}]{Fra61}}}] \label{Frankel} Let $(M^n,g)$ be a positively curved mani\-fold and let $N_1^{n_1}$ and $N_2^{n_2}$ be two connected, totally geodesic submanifolds. If $n_1 + n_2 \geq n$, then $N_1^{n_1}$ and $N_2^{n_2}$ intersect.
\end{thm}

The topology of $(M^n,g)$ is strongly reflected in the topology of a totally geodesic submanifold as shown in Wilking's Connectedness Lemma.

\begin{thm}[Connectedness Lemma {{\cite[\textit{Theorem 2.1, p.263}]{Wil03}}}] \label{connectedness} Let $(M^n,g)$ be a positively curved manifold.
\begin{enumerate}[1.)]

\item If $N^{n-k} \subset M^n$ is a compact totally geodesic submanifold of $M$, then the inclusion $N \hookrightarrow M$ is $(n-2k+1)$--connected. 

\item If $N_1^{n-k_1}, N_2^{n-k_2} \subset M^n$ are two compact totally geodesic submanifolds \mbox{of $M$} with $k_1 \leq k_2$ and $k_1 + k_2 \leq n$, then the inclusion $N_1 \cap N_2 \hookrightarrow N_2$ is \mbox{$(n-k_1-k_2)$}--connected.
\end{enumerate}
\end{thm}

We recall that a map $f: X \rightarrow Y$ is called $k$--connected, if the induced map on homotopy groups $f_{\ast}: \pi_i(X) \rightarrow \pi_i(Y)$ is an isomorphism for $i < k$ and an epimorphism for $i=k$.

Using Poincar\'e duality and the Hurewicz Theorem we note that a highly connected inclusion map $N^{n-k} \hookrightarrow M^n$ implies a periodicity on the integral cohomology ring of $M$. More precisely, we have

\begin{lem}[{{\cite[\textit{Lemma 2.2, p.264}]{Wil03}}}] \label{PD periodicity} Let $M^n$ and $N^{n-k}$ be two closed and oriented manifolds. If the inclusion $N^{n-k} \hookrightarrow M^n$ is $(n-k-l)$--connected with $n -k - 2l > 0$, then there exists \mbox{$e \in H^{k}(M; \Z)$} such that multiplication
$$\cup \; e: H^{i}(M; \Z) \rightarrow H^{i+k}(M; \Z) 
$$
is surjective for $l \leq i < n-k-l$ and injective for $l < i \leq n-k-l$.
\end{lem}

In this lemma the pullback of the class $e \in H^{k}(M; \Z)$ via the inclusion map is the Euler class of the normal bundle of $N$ in $M$. Wilking \cite{Wil03} applied this lemma to derive structure theorems for positively curved manifolds with large torus actions.

The Connectedness Lemma is  powerful, when two totally geodesic submanifolds intersect transversely. Let $N_1^{n-k_1}$ and $N_2^{n-k_2}$ be such as in the second part of the Connectedness Lemma and suppose that $N_1$ and $N_2$ intersect transversely. By Lemma \ref{PD periodicity} there exists $e \in H^{k_1}(N_2; \Z)$ such that 
$$\cup \; e: H^{i}(N_2; \Z) \rightarrow H^{i+k_1}(N_2; \Z)
$$
is surjective for $0 \leq i < n-k_1-k_2$ and injective for $0 < i \leq n-k_1-k_2$. Hence, $e \in H^{k_1}(N_2; \Z)$ generates a periodicity on the entire ring of $H^{\ast}(N_2; \Z)$. This leads to the following definition, which was introduced by Kennard \cite{Ken13}.

\begin{dfn}
Let $R$ be a ring and let $M^n$ be an oriented manifold. The cohomology ring $H^{\ast}(M;R)$ is said to be \textit{$k$--periodic}, if there exists $e \in H^{k}(M; R)$ such that $\cup \; e: H^{i}(M; R) \rightarrow H^{i+k}(M; R)$ is surjective for $0 \leq i < n-k$ and injective for $0 < i \leq n-k$.
\end{dfn}

Kennard studied $k$--periodic manifolds and showed by means of the Steenrod power operations that integrally $k$--periodic manifolds are rationally 4--periodic. In combination with the Connectedness Lemma, Kennard proved the next result.

\begin{thm}[Periodicity Theorem {{\cite[\textit{Theorem 4.2, p.578}]{Ken13}}}] \label{periodicity theorem} Let $(M^n,g)$ be a simply-connected, positively curved mani\-fold. Let $N_1^{n-k_1}, N_2^{n-k_2} \subset M^n$ be totally geodesic submanifolds that intersect transversely. If $2k_1 + 2k_2 \leq n$, then $M$ is rationally $4$--periodic.
\end{thm}

If the dimension of $(M^n,g)$ is divisible by four, then it follows that the rational cohomology ring of a simply-connected, $4$--periodic manifold is generated by a single class. Hence, $H^{\ast}(M;\Q)$ is isomorphic to the rational cohomology ring of a sphere, a complex projective space or a quaternionic projective space.

The following lemma states that the periodicity is preserved by highly connected maps. It follows from the Connectedness Lemma and a similar version appears in the proof of the Periodicity Theorem (see the proof of the theorem's second part {{\cite[\textit{Theorem 4.2, p.578}]{Ken13}}}). Hence, we omit the proof of the next lemma.

\begin{lem} \label{key lemma}
Let $(M^n,g)$ be a closed, positively curved manifold of even dimension with $n \geq 8$. Let $N^{n-k} \subset M^n$ be a totally geodesic submanifold of even codimension with $k \leq \frac{n}{4}$. If $N$ is rationally $4$--periodic, then so is $M$.
\end{lem}

\subsection{Elliptic genus} We move on to the topological part of this paper. We  give a brief description of the elliptic genus and mention its different expansions as well as its rigidity properties with regard to compact Lie group actions. For an introduction to the subject, the reader is referred to  \cite{HBJ92}, \cite{HS90} and \cite{Lan88}.

A genus in the sense of Hirzebruch is a ring homomorphism from the oriented cobordism ring $\Omega_{\textrm{SO}}^\ast \otimes \Q$ to a commutative, unital $\Q$-algebra $R$. As examples, we mention the signature and the $\hat{A}$--genus. As these examples suggest, genera often arise in the context of Index Theory. The Hirzebruch formalism describes a correspondence between genera and power series $Q(x)$ with coefficients in $R$. The \textit{elliptic genus} $\phi(M)$ is the genus associated to the power series $Q(x) = x/f(x)$ with
\begin{equation*}
f(x) = \frac{1-e^{-x}}{1+e^{-x}} \prod_{n=1}^{\infty} \frac{1-q^ne^{-x}}{1+ q^ne^{-x}} \cdot \frac{1-q^n e^{x}}{1+q^n e^{x}}.
\end{equation*}
Since the function $f(x)$ is attached to a certain lattice, this yields a close relation between the elliptic genus $\phi(M)$ and the theory of modular forms. For instance, we note that for $n = 0 \pmod 8$ the elliptic genus $\phi(M^n)$ is a modular function for the subgroup $\Gamma_0(2) \subset \textrm{SL}_2(\Z)$.

According to Witten \cite{Wit88}, the elliptic genus $\phi(M)$ admits a remarkable interpretation. It can be thought of as the equivariant signature of the free loop space $\mathcal{L}M$ with respect to the natural circle action on $\mathcal{L}M$. We recall that $S^1$ acts on the loop space by reparametrizing the loops. This approach yields the following power series for the elliptic genus
\begin{eqnarray} \label{sign desc}
\phi(M) &=& sign(M, \bigotimes_{n=1}^{\infty}S_{q^n}T_\C M \otimes \bigotimes_{n=1}^{\infty}\Lambda_{q^n}T_\C M) \\
&=& sign(M) + 2 \; sign(M, T_\C M)\cdot q + \ldots, \nonumber
\end{eqnarray} 
where  $T_\C M$ denotes the complexified tangent bundle and 
$$S_{t}T_\C M = \sum_{i=0}^{\infty} S^{i}T_\C M \cdot t^{i} \quad \textrm{and} \quad \Lambda_{t}T_\C M = \sum_{i=0}^{\infty} \Lambda^{i}T_\C M \cdot t^{i}.
$$

It follows that in the cusp given by $q=0$ the elliptic genus equals the signature. In the other cusp, $\phi(M)$ can be described in terms of twisted $\hat{A}$-genera. More precisely, $\phi(M)$ has the following $q$-development
\begin{eqnarray} \label{spin desc}
\phi_0(M^{4k}) &=& q^{-k/2} \cdot \hat{A}(M, \bigotimes_{\substack{n \geq 1 \\ n \; \textrm{odd}}}^{\infty}\Lambda_{-q^n} T_{\mathbb{C}}M \otimes \bigotimes_{\substack{n \geq 1 \\ n \; \textrm{even}}}^{\infty} S_{q^n} T_{\mathbb{C}}M)\\
&=& q^{-k/2}\cdot(\hat{A}(M) - \hat{A}(M,T_{\mathbb{C}}M)\cdot q \; \pm \ldots). \nonumber
\end{eqnarray} 

If $M$ is Spin, then $\hat{A}(M, W)$ can be geometrically seen as the index of the Dirac operator twisted with some complex vector bundle $W$. In this case, the coefficients of the power series (\ref{spin desc}) are integers. We conclude that (\ref{sign desc}) and (\ref{spin desc}) reveal a wonderful connection between Spin and signature geometry.

\subsection{Rigidity property} We turn our attention to the equivariant setting. Let $G$ be a compact, connected Lie group acting on a Spin manifold $M$.  Then, the associated vector bundles in (\ref{sign desc}) become $G$-bundles and the elliptic genus $\phi(M)$ refines to an equivariant genus $\phi(M)_g$ depending on $g\in G$. However, Bott and Taubes \cite{BT89} proved  that the elliptic genus is \textit{rigid}, i.e. 
$$\phi(M)_g = \phi(M), \quad \forall g \in G.
$$

Suppose now that $S^1$ acts on a Spin manifold $M$ and let $\sigma \in S^1$ be the non-trivial involution. Using the rigidity property and the Lefschetz fixed point formula  of Atiyah-Segal-Singer and Bott, it was shown by Hirzebruch and Slodowy \cite{HS90} that
$$ \phi(M) = \phi(M)_{\sigma} = \phi(M^{\sigma} \circ M^{\sigma}),
$$
where $M^{\sigma} \circ M^{\sigma}$ denotes the transverse self-intersection of the fixed point manifold $M^{\sigma}$. By studying the order of the pole in expression (\ref{spin desc}), Hirzebruch and Slodowy deduced the following result, which will play a crucial role in our arguments.

\begin{thm}[{{\cite[\textit{Corollary, p.317}]{HS90}}}] \label{thm_HS} Let $S^1$ act on a Spin manifold $M^{4k}$. If the action is odd, then $\phi(M) =0$. If the action is even and codim $M^{\sigma} > 4r$, then the first $(r+1)$ coefficients of $\phi_0(M)$ vanish.
\end{thm}

We recall that a circle action is called \textit{even}, if the circle action lifts to the Spin structure. If this is not the case, then the action is called \textit{odd}. Hirzebruch and Slodowy  used this theorem to show that any Spin homogeneous space has constant elliptic genus.

\section{Elementary coding theory}

In this section we turn our attention to symmetry. In order to analyze the action of involutions, we make use of the theory of error-correcting codes. More precisely, we attach to a given torus fixed point a binary linear code reflecting the symmetry around this fixed point. In particular, the code measures both the size and the position of the various fixed point components around the fixed point. This approach has often been used  in the study of large group actions on smooth manifolds such as in \cite{Wil03}. We explain some elementary coding theory here to have a systematic presentation of the symmetry. For an introduction to coding theory we refer to the book \cite{HP03}.

\subsection{Linear codes} We consider the finite dimensional vector space $\Z_2^n$. A \textit{binary $[n,k]$--linear code} $\mathcal{C}$ of length $n$ and rank $k$ is a $k$--dimensional subspace of $\Z_2^n$. A linear code $\mathcal{C}$ may be presented by a \textit{generating matrix} $G$, where the rows of $G$ form a basis of the code $\mathcal{C}$. Moreover, the code $\mathcal{C}$ comes with the \textit{Hamming distance} function defined by
$$ d(x,y) := \textrm{ord} (\{ i \; \vert \; x_i \neq y_i \}) \quad \textrm{for} \; x,y \in \mathcal{C}.
$$ 
We define the \textit{weight} $\textrm{wt}(x)$ of a codeword $x\in \mathcal{C}$ to be the number of coordinates that are non-zero, i.e. $\textrm{wt}(x) = d(x,0)$. Eventually, we define the \textit{minimum} and the \textit{maximum distance} of a code $\mathcal{C}$ to be
$$ d_{\textrm{min}}(\mathcal{C}) = \min_{ x \in \mathcal{C}^{\ast} } \; \mathrm{wt}(x) \quad \textrm{and} \quad d_{\textrm{max}}(\mathcal{C}) = \max_{x \in \mathcal{C}^{\ast}} \; \mathrm{wt}(x).
$$
\indent We shall now derive some properties of the maximum distance by using the construction of the residual code. Let $\mathcal{C}$ be an $[n,k]$--linear code and let $d_{\textrm{max}}(\mathcal{C})$ be the maximum distance. Let $x \in \mathcal{C}$ be a codeword that realizes the maximum distance. We construct the residual code $\mathcal{C}^{\textrm{res}}(x)$  with respect to $x \in \mathcal{C}$ in the following manner.

First, we choose a generating matrix $G$  of $\mathcal{C}$, whose first row corresponds to $x \in \mathcal{C}$. Then, we permute the non-zero entries of $x \in \mathcal{C}$ to the front and write down the new generating matrix
$$
\left( \begin{array}{c|c}
1 \; 1 \ldots 1 & 0 \; 0 \ldots 0 \\
\hline
\rule{0pt}{2.8ex} 
G' & G''
\end{array} \right)
$$
for the code $\mathcal{C}$.

\begin{dfn}
In the construction above, we define the \textit{residual code} $\mathcal{C}^{\textrm{res}}(x)$ \textit{with respect to} $x \in \mathcal{C}$ to be the linear code generated by the matrix $G''$.
\end{dfn}

It follows from the construction that the residual code $\mathcal{C}^{\textrm{res}}(x)$ is a linear code of length $n- \textrm{wt}(x)$ and is of rank at most $n-1$. For our considerations, it is important to note that any codeword in the residual code $\mathcal{C}^{\textrm{res}}(x)$ comes initially from a codeword in $\mathcal{C}$. Our definition of the residual code agrees with the standard definition \cite{HP03}, but where we fixed $x \in \mathcal{C}$ to have maximum distance. We prove a simple estimate for the maximum distance of this new code. It is merely a variant of the Griesmer step, which is used in the proof of the Griesmer bound.

\begin{lem} \label{dmax}
Let $\mathcal{C}$ be a binary linear code. Then, the maximum distance of the residual code with respect to $x \in \mathcal{C}$ satisfies $ d_{\textrm{max}}(\mathcal{C}^{\textrm{res}}(x)) \leq \lfloor d_{\textrm{max}}(\mathcal{C}) /2 \rfloor$.
\end{lem}

\begin{proof}
Let $y \in \mathcal{C}$ be a codeword.  We decompose the codeword with respect to $G'$ and $G''$ and denote the splitting by $y = (y_1 \; \vert \; y_2)$. By construction, we have $y_2 \in \mathcal{C}^{\textrm{res}}(x)$ and we need to show that
$$ \textrm{wt}(y_2) \leq \left \lfloor \frac{d_{\textrm{max}}(\mathcal{C})}{2} \right \rfloor.
$$
Since $y \in \mathcal{C}$, we observe that
$$  \textrm{wt}(y_1)+ \textrm{wt}(y_2) \leq d_{\textrm{max}}(\mathcal{C}).
$$
Moreover, the weight of the codeword $x+y \in \mathcal{C}$ implies
$$ d_{\textrm{max}}(\mathcal{C}) - \textrm{wt}(y_1) + \textrm{wt}(y_2) \leq d_{\textrm{max}}(\mathcal{C}).
$$
We add up these two inequalities to obtain
$$\textrm{wt}(y_2) \leq \left \lfloor \frac{d_{\textrm{max}}(\mathcal{C})}{2} \right \rfloor,
$$
which is the claim.
\end{proof}

The construction of the residual code can be iterated. Let $\mathcal{C}$ be a linear code of length $n$ and let $\mathcal{C}^{\textrm{res}}(x_1, \ldots, x_{k-1})$ be the residual code with respect to the codewords $x_i \in \mathcal{C}^{\textrm{res}}(x_1, \ldots, x_{i-1})$ for $2 \leq i \leq k-1$. We choose a codeword $x_k \in  \mathcal{C}^{\textrm{res}}(x_1, \ldots, x_{k-1})$ of maximal weight and construct the new residual code $\mathcal{C}^{\textrm{res}}(x_1, \ldots, x_{k})$ as above. This code is of length $n- \sum_{i=1}^k \textrm{wt}(x_i)$ and by  \mbox{Lemma \ref{dmax}}, we note that
$$ d_{\textrm{max}}(\mathcal{C}^{\textrm{res}}(x_1, \ldots, x_{k})) \leq \left \lfloor \frac{d_{\textrm{max}}(\mathcal{C})}{2^k} \right \rfloor .
$$
We summarize this observation into the following lemma.

\begin{lem} \label{iter_res_cod}
Let $\mathcal{C}$ be a binary linear code of length $n$ and  of maximum distance $d_{\textrm{max}}(\mathcal{C}) \leq \frac{n}{2}$ and $d_{\textrm{max}}(\mathcal{C}) < 2^{k}$. Then, $\mathcal{C}^{\textrm{res}}(x_1, \ldots, x_{k})$ is the trivial code of length $n- \sum_{i=1}^k \textrm{wt}(x_i) >0$.
\end{lem}

\subsection{Group actions on manifolds} For the remainder of this section, we will discuss the implementation of coding theory into the framework of positively curved manifolds with symmetry. The aim of this approach is to gain an overview on the position of the fixed point sets and to keep track of the dimensions.

Let $(M^{2n},g)$  be an oriented, positively curved manifold of even dimension. Suppose that a torus $T^s$ acts isometrically and effectively on $M$ and let $\textit{pt} \in M^T$ denote a torus fixed point.  As such, the torus acts linearly on the tangent space $T_{\textit{pt}}M$ and hence, we obtain the isotropy representation
$$ \hat{\rho}: T^s \longrightarrow \textrm{SO}(T_{\textit{pt}}M).
$$
Next, we pass to the subgroup of involutions $\Z_2^s \subset T^s$ and we note that $\hat{\rho}$ induces an embedding
$$ \rho: \Z_2^s\longrightarrow \Z_2^n.
$$
Thus, we attach the binary linear code $\mathcal{C} := \rho(\Z_2^s)$ to the fixed point $\textit{pt}\in M^T$. Since the torus action is effective, the induced homomorphism $\rho$ is injective. Therefore, we conclude that $\mathcal{C}$ is an $[n,s]$--linear code.

\begin{dfn}
In the above situation, we consider an involution $\sigma \in T^s$. The image $\tilde{\sigma} := \rho(\sigma) \in \mathcal{C}$ is called the \textit{codeword associated to the involution} $\sigma \in T^s$.
\end{dfn}

The resulting code $\mathcal{C}$ captures, in particular, the dimension of the fixed point set of involutions. Let $\sigma \in T^s$ be an involution and let $F(\sigma) \subset M^{\sigma}$ be the fixed point component containing $\textit{pt} \in M^{T}$.  Then, we recognize that the weight $\textrm{wt}(\tilde{\sigma})$ of the associated codeword measures exactly one half of the codimension of $F(\sigma)$ in $M$. Moreover, the dimension of the intersection of various fixed point components around $\textit{pt} \in M^T$ can be easily read off from the code.

\section{Proofs of the results}

\subsection{Proofs of Theorems \ref{thmA} and \ref{thmB}}

This subsection deals with the proofs of Theorem \ref{thmA} and \ref{thmB}. In both proofs, we examine closely the fixed point components of the involutions in $T^s$. The main idea is that, if these totally geodesic submanifolds have high codimension, then the coefficients of $\phi_0(M)$ vanish, whereas, if the codimensions are low, we are able to compute the rational cohomology ring of $M$. Before we begin the proof, we fix the following notations.

\begin{nota}
Let $N \subset M$ be a connected submanifold. Then $\cod_{M} N$ denotes the codimension of $N$ in $M$.
\end{nota}

\begin{nota}
Suppose that a torus $T^s$ acts smoothly on a manifold $(M^n,g)$ and let $\sigma \in T^s$ be an involution. We fix a torus fixed point $\textit{pt} \in M^T$ and set $F(\sigma)$ to be the fixed point component of $M^{\sigma}$ containing $\textit{pt} \in M^T$. Furthermore, let $\sigma_1, \sigma_2 \in T^s$ be two involutions. We consider the induced action of $\sigma_2$ on $F(\sigma_1)$ and write $F(\langle \sigma_1, \sigma_2 \rangle)$ for the fixed point set at $\textit{pt}$. We notice that $F(\langle \sigma_1, \sigma_2 \rangle)$ is the component of the intersection \mbox{$F(\sigma_1) \cap F(\sigma_2)$} at $\textit{pt}$.
\end{nota}

We now state the key technical lemma. Under certain assumptions on the torus action, we construct a chain of submanifolds of decreasing codimensions. These submanifolds are obtained through the intersection of fixed point sets of involutions. In the construction of this chain, we also come across submanifolds that intersect transversely and therefore, we are able to detect a periodicity in the cohomology.  At this point, we  use the coding-theoretic interpretation of symmetry.

\begin{lem} \label{4per_fam}
Let $(M^n,g)$ be an even-dimensional, closed, oriented, positively curved manifold and let $s \geq 2$. Suppose that a torus $T^{s-1}$ acts isometrically and effectively on $M^n$ and let $ \textit{pt} \in M^T$ be a torus fixed point. Suppose that $\cod_M F(\sigma) \leq \frac{n}{2}$ and $\cod_M F(\sigma) < 2^{s-1}$ for all involutions $\sigma \in T^{s-1}$. Then, we have the following statements:
\begin{enumerate}[1.)]
\item There exists a family of totally geodesic submanifolds
\begin{equation*}
N_{s-1} \subseteq N_{s-2} \subseteq \ldots \subseteq N_{1} \subset N_0 = M
\end{equation*} 
such that  $\cod_{N_i} N_{i+1} \leq \frac{1}{2} \; \cod_{N_{i-1}}N_i$ for all $1 \leq i \leq s-2$ and $\textrm{dim} \; N_{s-1} > 0$.
\vspace{0.15cm}

\item In the family above, we suppose that $\cod_{N_1} M = k$. Then, each inclusion map $N_{i+1} \hookrightarrow N_i$ is at least $(n-2k+1)$--connected.

\vspace{0.15cm}
\item In the family above, there exists $r \in \{ 0, \ldots, s-1 \}$ such that $N_r$ is rationally \mbox{$4$--periodic}.
\end{enumerate}
\end{lem}

\begin{proof}
\textit{Ad 1.}: First, we translate the symmetry into an error-correcting code and afterwards, we make use of the residual code to construct a family of submanifolds with the desired properties.

The isotropy representation at the point $\textit{pt}\in M^T$ yields a linear code $\code$ of length $\frac{n}{2}$ and of rank $s-1$.  Starting with this code $\code$, we construct a residual code $\code^{\textrm{res}}(\tilde{\tau}_1, \ldots, \tilde{\tau}_{s-1})$ and recall that each codeword in $\code^{\textrm{res}}(\tilde{\tau}_1, \ldots, \tilde{\tau}_{i})$ comes, by construction, from a codeword in $\mathcal{C}$. We then define the submanifolds $N_i$ in correspondence to these codewords.

First, we assume that the codewords $\tilde{\tau}_{i+1} \in \code^{\textrm{res}}(\tilde{\tau}_1, \ldots, \tilde{\tau}_{i})$ are non-zero.  Let $\sigma_1 \in T^{s-1}$ be the involution associated to the codeword $\tilde{\tau}_1 \in \code$. We put
$$ N_1 = F(\sigma_1)
$$
and consider then the induced torus action on $N_1$. In the next step, we choose the involution $\sigma_2 \in T^{s-1}$ such that the associated codeword $\tilde{\sigma}_2 \in \mathcal{C}$ gives the codeword $\tilde{\tau}_2   \in \code^{\textrm{res}}(\tilde{\tau}_1)$. We set
$$ N_2 = F(\langle \sigma_1, \sigma_2 \rangle)
$$
and from here, we proceed inductively. Let $N_i = F(\langle \sigma_1, \ldots, \sigma_i \rangle)$ be given. We select the involution $\sigma_{i+1} \in T^{s-1}$ such that the associated codeword $\tilde{\sigma}_{i+1} \in \mathcal{C}$ gives the codeword $\tilde{\tau}_{i+1} \in \code^{\textrm{res}}(\tilde{\tau}_1, \ldots, \tilde{\tau}_{i})$ and we define
$$ N_{i+1} = F(\langle \sigma_1, \ldots, \sigma_{i+1} \rangle).
$$
Since the initial code $\code$ has rank $s-1$,  this construction stops at the submanifold $N_{s-1}$. Therefore, we have built a family  of totally geodesic submanifolds.

Now let $\tilde{\tau}_{i+1} \in \code^{\textrm{res}}(\tilde{\tau}_1, \ldots, \tilde{\tau}_{i})$ be the first codeword among $(\tilde{\tau}_1, \tilde{\tau}_2, \ldots, \tilde{\tau}_{s-1})$ that is trivial and, as above, let $N_i = F(\langle \sigma_1, \ldots, \sigma_i \rangle)$ be given. Then, we complete our family by putting
$$ N_{s-1} = N_{s-2} = \ldots = N_{i+1} = N_i.
$$

By Lemma \ref{dmax}, the codimensions  decrease by a factor of at least one half yielding
$$ \cod_{N_i} N_{i+1} \leq \frac{1}{2} \; \cod_{N_{i-1}}N_i \quad \textrm{for} \quad  1 \leq i \leq s-2.
$$
Moreover, since  $\cod_M F(\sigma) \leq \frac{n}{2}$, we remark that $\dim N_{s-1} >0$.
\medskip

\textit{Ad 2.}: The proof is a consequence of the Connectedness Lemma. We recall that by construction the codimensions decrease by a factor of at least one half
$$ \cod_{N_i} N_{i+1} \leq \frac{1}{2} \; \cod_{N_{i-1}}N_i. \
$$
We consider the map $N_{i+1} \hookrightarrow N_i$ for some $i \in \{ 1, \ldots, s-2 \}$. By our construction, we observe that 
\begin{eqnarray*}
\dim N_i - 2 \cdot \cod_{N_i} N_{i+1} +1 &=& \dim N_{i-1} - \cod_{N_{i-1}} N_{i} - 2 \cdot \cod_{N_i} N_{i+1} +1 \\
&\geq & \dim N_{i-1} - 2 \cdot \cod_{N_{i-1}} N_{i} +1.
\end{eqnarray*}
It follows from the Connectedness Lemma that the inclusion map $N_{i+1} \hookrightarrow N_i$ is as connected as the map $N_{i-1} \hookrightarrow N_i$. In other words, the inclusion maps in our family preserve the connectivity. Since $N_1 \hookrightarrow M$ is $(n-2k+1)$--connected, the claim follows directly.
\medskip

\textit{Ad 3.}: The strategy is to show that the induced action on the submanifolds $N_i$ becomes ineffective at some stage.  Consequently, this will lead to a rationally $4$--periodic manifold in the family.

We proceed via contradiction and assume that the submanifold $N_i$ is not rationally $4$--periodic for any $i \in \{ 0, \ldots, s-1 \}$. Let $\Z_2^{s-1} \subset T^{s-1}$ be the subgroup of involutions.

First, we show that there exists a group $\Z_2^{s-2}$ acting  effectively on $N_1 =F(\sigma_1)$. In fact, if the kernel of the $\Z_2^{s-1}$--action on $N_1$ has rank greater than one, then we have  non-trivial involutions $\rho, \sigma_1 \in \Z_2^{s-1}$ acting trivially on $N_1$. Therefore, the components $F(\rho)$ and $F(\rho \cdot \sigma_1)$ intersect transversely at $\textit{pt} \in M^T$, the intersection being $F(\sigma_1)$. By assumption, the codimensions satisfy
$$ 2 \cdot  \cod_{M} F(\rho \cdot \sigma_1) + 2 \cdot  \cod_{M} F(\rho) = 2 \cdot \cod_{M} F(\sigma_1) \leq n
$$
and therefore, $M$ would be rationally $4$--periodic by Theorem \ref{periodicity theorem}. However, this leads to a contradiction and hence, we have an effective $\Z_2^{s-2}$--action on $N_1$.

In the same way, we show by an inductive argument that there is a group $\Z_2^{s-i-1}$ acting effectively on $N_i$ for $i \in \{ 1, \ldots, s-2 \}$. In fact, the codimensions are small enough
$$ 2 \cdot \cod_{N_i} N_{i+1} \leq \dim N_i
$$
so that we can use Theorem \ref{periodicity theorem}, when the kernel of the induced $\Z_2^{s-i-1}$--action from $N_i$ restricted to $N_{i+1}$ has rank greater than one. Therefore, we obtain an effective $\Z_2$--action on $N_{s-2}$.

If we rephrase the last statement in terms of coding theory, we note that the residual code $\code^{\textrm{res}}(\tilde{\tau}_1, \ldots, \tilde{\tau}_{s-2})$ is non-trivial. By assumption, we have \mbox{$d_{\max}(\code) < 2^{s-2}$} and so it follows by Lemma \ref{iter_res_cod} that the  code $\code^{\textrm{res}}(\tilde{\tau}_1, \ldots, \tilde{\tau}_{s-2})$ is trivial. Hence, the $\Z_2$--action on $N_{s-2}$ is trivial and we have established a contradiction. This implies that there is a $r \in \{ 0, \ldots, s-1 \}$ such that the manifold $N_{r}$ is rationally $4$--periodic.
\end{proof}

We are now ready to prove Theorem \ref{thmA}. Roughly, the idea goes as follows. The non-vanishing of the elliptic genus imposes restrictions on the fixed point set of involutions. It turns out that the fixed point configurations are compatible with the assumptions of Lemma \ref{4per_fam}. So, we can find a rationally $4$--periodic submanifold, which induces a periodicity on the ambient manifold.

\begin{proof}[Proof of Theorem \ref{thmA}]
The proof combines the various aspects from topology, geometry and symmetry that we developped so far.  

First, we set up the proof. The statement is obviously true, if the dimension is not divisible by four. The first coefficient of the power series $\phi_0(M)$ is the \mbox{$\hat{A}$--genus} and the statement is true by the Lichnerowicz Theorem for dimensions $n \leq 12$, even without any symmetry condition. So let $n \geq 16$. We work under the assumption that the first $\min \{ \left \lfloor \frac{n}{16} \right \rfloor +1, 2^{s-3} \}$ coefficients of $\phi_0(M)$ do not vanish and we show that $M$ is rationally $4$--periodic.

Under this assumption, we conclude  with Theorem \ref{thm_HS} that any involution $\sigma \in T^s$ has a connected fixed point component $F \subset M^{\sigma}$ with
$$ \cod_{M} F \leq 4 \left \lfloor \frac{n}{16} \right \rfloor \leq \frac{n}{4} \quad \textrm{and} \quad \cod_{M} F < 2^{s-1}.
$$
Furthermore, we may assume that the action is even and so the codimension $\cod_{M}F$ is divisible by four. The torus action  is isometric and so all connected components of $M^{\sigma}$ are totally geodesic submanifolds. Hence, the Intersection Theorem \ref{Frankel} restricts the codimensions of fixed point components $\tilde{F} \subset M^{\sigma}$ of involutions to
$$ \cod_M \tilde{F} \; \leq \; \frac{n}{4} \quad \textrm{or} \quad \cod_M \tilde{F} \; \geq \; \frac{3}{4} \; n +4.
$$
Since $\phi_0(M)$ is non-zero, there exists a torus fixed point $\textit{pt} \in M^T$. As before, we describe the symmetry around $\textit{pt}$ in terms of a linear code $\code$. As such, the code has length $\frac{n}{2}$ and is of rank $s$. Let $F(\sigma)$ be again the fixed point component of the involution $\sigma \in T^s$ around \textit{pt} and let $\tilde{\sigma} \in \code$ be the associated codeword. Following the fixed point configuration, we obtain
$$ \textrm{wt}(\tilde{\sigma}) \; \leq \; \frac{n}{8} \quad \textrm{or} \quad \textrm{wt}(\tilde{\sigma}) \;  \geq \;  \frac{3}{8} \; n+2.
$$

Next, we will see that there is a subgroup $\Z_2^{s-1} \subset T^s$, for which all involutions have $\cod_M F(\sigma) \leq \frac{n}{4}$. This is equivalent to showing that the subcode
$$ \code' := \left \{ \tilde{\sigma} \in \code \; \vert \; \textrm{wt}(\tilde{\sigma}) \leq \frac{n}{8} \right \}
$$
has at least rank $s-1$.  It is clear that $\code'$ is closed under addition and therefore, yields a linear subspace of $\code$. In order to get an estimate on the rank of $\code'$, we choose a linear code $\code''$ such that $\code$ is the direct sum of $\code'$ and $\code''$. 
Let $c_1, c_2 \in \code''$ be two non-zero codewords. By definition of $\code'$, we easily see that
$$ \textrm{wt}(c_1 +c_2) \;< \; \frac{3}{8} \; n+2.
$$
Subsequently, we have that $c_1 + c_2 \notin \code''$ and so $\code''$ has at most rank one. It follows that the rank of $\code'$ is at least $s-1$.  We conclude that a subgroup $\Z_2^{s-1} \subset T^s$ acts isometrically  on $M$ such that
$$\cod_M F(\sigma) \leq \frac{n}{4} \quad \textrm{and} \quad \cod_M F(\sigma) < 2^{s-1}
$$
for any involution $\sigma \in \Z_2^{s-1}$.

We are now in the fortunate position to use Lemma \ref{4per_fam}. There exists a family of totally geodesic submanifolds
$$ N_{s-1} \subseteq N_{s-2} \subseteq \ldots \subseteq N_1 \subset M
$$
with $N_r$ a rationally $4$--periodic manifold for some $r \in \{0, \ldots, s-1 \}$. Moreover, the codimensions satisfy
$$ \cod_{N_i} N_{i+1} \leq \frac{1}{2^i} \; \cod_M N_1 \leq \frac{n}{2^{i+2}}.
$$
In particular, we have that $\dim N_i \geq \frac{n}{2} \geq 8$ for each $i \in \{ 1, \ldots, s-1 \}$.

We finish the proof by applying successively Lemma \ref{key lemma} to our family of totally geodesic submanifolds. Since $N_r$ is rationally $4$--periodic and
$$ 4 \cdot \cod_{N_{r-1}} N_r \leq \frac{n}{2^{r-1}} \leq \frac{n}{2} \leq \dim N_{r-1} \quad \textrm{for} \; r \geq 2,
$$ 
we conclude that $N_{r-1}$  is also rationally $4$--periodic.  Proceeding inductively, we note eventually that $N_1$ is rationally $4$--periodic. Since the codimension satisfies $\cod_{M} N_1 \leq \frac{n}{4}$, Lemma \ref{key lemma} finally implies that $M$ is rationally \mbox{$4$--periodic}. This marks the end of the proof.
\end{proof}

The proof of Theorem \ref{thmB} is centered around the same ideas as the proof of Theorem \ref{thmA}.

\begin{proof}[Proof of Theorem \ref{thmB}]
The setup is essentially the same as before. The second statement is true for $n \leq 8$ by the Lichnerowicz Theorem and for $n \leq 20$ by Dessai's work \cite{Des07}. Let $n \geq 24$ and assume that the first $\min \{ \left \lfloor \frac{n}{12} \right \rfloor +1, 2^{s-3} \}$ coefficients of $\phi_0(M)$ do not vanish.  We may assume that the action is even.

Using Theorem \ref{thm_HS}, we obtain restrictions on the fixed point set of involutions. Indeed, any involution $\sigma \in T^s$ has a fixed point component $F \subset M^{\sigma}$ such that
$$
\cod_{M} F \leq \frac{n}{3} \quad  \textrm{and} \quad \cod_M F < 2^{s-1}.
$$

We will show that there exists an involution $ \sigma_1 \in T^s$ such that $F \subset M^{\sigma_1}$ is rationally $4$--periodic with $\cod_M F \leq \frac{n}{3}$. The non-vanishing of the elliptic genus ensures that the fixed point set $ M^T$ is non-empty. We take a torus fixed point $\textit{pt} \in M^T$. As in the proof of Theorem \ref{thmA}, we use the Intersection Theorem  to show that there is a subgroup $\Z_2^{s-1} \subset T^s$, for which all involutions $\sigma \in \Z_2^{s-1}$ have $\cod_M F(\sigma) \leq \frac{n}{3}$ and $\cod_M F(\sigma) < 2^{s-1}$ around the fixed point $\textit{pt}$. Hence, we can apply Lemma \ref{4per_fam} to exhibit a family of totally geodesic submanifolds 
$$ N_{s-1} \subseteq N_{s-2} \subseteq \ldots \subseteq N_1 \subset M
$$
with $N_r$ a rationally $4$--periodic manifold for some $r \in \{ 0, \ldots, s-1 \}$ and with the property that
\begin{equation} \label{cods}
\cod_{N_i} N_{i+1} \; \leq \; \frac{1}{2^i} \; \cod_M N_1 \; \leq \; \frac{n}{3 \cdot 2^{i}}.
\end{equation}
In particular, we have that $\dim N_i \geq \frac{n}{3} \geq 8$ for each $i \in \{ 1, \ldots, s-1 \}$.

In order to invoke Lemma \ref{key lemma}, we need to check the codimensions more precisely.  From (\ref{cods}), we remark that
$$ 4 \cdot \cod_{N_1} N_2 \; \leq \frac{2}{3} n \leq \dim N_1
$$
and also that
$$ 4 \cdot\cod_{N_i} N_{i+1} \leq \frac{n}{3 \cdot 2^{i-2}} \leq \frac{n}{3}  \leq \dim N_i \quad \textrm{for} \; i \geq 2.
$$

We conclude that the codimensions are small enough to use Lemma \ref{key lemma}. As a result, the $4$--periodicity on $N_r$ induces a $4$--periodicity on $N_1$. We recall that by construction, $N_1$ is the fixed point component attached to some involution $\sigma_1 \in T^s$. Since the action is even,   $\dim N_1$ is divisible by four and hence, $N_1$ is a rational cohomology sphere, complex or quaternionic projective space.

We rule out the case of $N_1$ being a rational cohomology sphere. Let $N_1 =F(\sigma_1)$ be a rational cohomology sphere with $\cod_M N_1 \leq \frac{n}{3}$. By the Lefschetz fixed point formula, the elliptic genus can be computed in terms of the fixed point set $M^{\sigma_1}$, where the Euler classes of the normal bundles factor out.  The Intersection Theorem implies, however, that $M^{\sigma_1}$ consists  of the cohomology sphere $N_1$ and of components of small dimension and consequently, the Euler classes of the normal bundles vanish. Therefore,  the elliptic genus vanishes, which is a contradiction.

We use the Connectedness Lemma and its implications to wrap up the proof. The inclusion map $N_1 \hookrightarrow M$ is $(\frac{n}{3}+1)$--connected. We fix $k_1 := \cod_{M} N_1$ to be the codimension of $N_1$ in $M$, which is divisible by four. We note that there exists a class $e \in H^{k_1}(M ; \Q)$ such that multiplication
\begin{equation} \label{thmB_2}
\cup e : H^{i}(M; \Q) \rightarrow H^{i+k_1}(M;\Q)
\end{equation} 
is an isomorphism for $k_1 \leq i \leq n-2k_1$.

We recall that the inclusion map is highly connected and that $N_1$ is rationally $4$--periodic. So we choose $x \in H^{4}(M; \Q)$ to be a generator and note that $e = x^{k_1/4}$ up to some non-zero constant. Finally, by (\ref{thmB_2}) there exists an integer $m \in \N$ such that $e^m \neq 0 \in H^{m \cdot k_1}(M; \Q)$ with $m \cdot k_1 \geq \frac{2}{3}n$ and so that $x^{m \cdot k_1 /4}$ generates $H^{m \cdot k_1}(M; \Q)$. By the cup product version of Poincar\'e duality, we have that $x^{n/4}$ is non-zero, which is exactly the content of first statement.
\end{proof}

\subsection{Proof of Theorem \ref{thmC}}

It follows from the celebrated Smith theory that a smooth circle action on a sphere $S^{2n}$ comes either with a connected fixed point set or with a pair of isolated fixed points. In this subsection we study these  fixed point setups in perspective of positive curvature and the elliptic genus.

For the proof of Theorem \ref{thmC}, we proceed in a similar way as for the previous theorems. The fixed point configurations are particularly nice, since the torus fixed point set is connected. Eventually, we show that the elliptic genus is constant via the fixed point formula.

\begin{proof}[Proof of Theorem \ref{thmC}] 
The proof goes by contradiction. Our symmetry assumption implies that we have a torus $T^s$ acting isometrically on $M^n$ with $n \leq 2^s$. We suppose that the elliptic genus is not constant and without loss of generality, we may assume that the dimension of $M$ is divisible by four and that the action is even. Since the torus fixed point set $M^T$ is connected, we have that $M^{\sigma}$ is also connected for each involution $\sigma \in T^s$. For if $X_1, X_2 \subset M^{\sigma}$ were two fixed point components, with say $M^T \subset X_1$, then the induced torus action  on $X_2$ would not have any fixed points. However, $X_2$ is positively curved and even-dimensional. So, we obtain a contradiction with Berger's Lemma \cite{Ber66}. Subsequently, $M^{\sigma}$ is also connected.

Next, we conclude by Theorem \ref{thm_HS} that for any involution $\sigma \in T^s$ the fixed point set $F(\sigma) = M^{\sigma}$ satisfies
$$ \cod_{M} F(\sigma) \leq \frac{n}{2} - 2 < 2^{s-1}.
$$
We pick a fixed point $\textit{pt} \in M^T$. In view of the fixed point configuration, we apply Lemma \ref{4per_fam} to obtain a family of totally geodesic submanifolds
$$ N_{s-1} \subseteq \ldots \subseteq N_1 \subset N_0 = M
$$
such that $N_r$ is a rationally $4$--periodic manifold for some $r \in \{ 0, \ldots, s-1 \}$.  By construction, we have $N_1 =F(\sigma_1)$ for an involution $\sigma_1 \in T^s$.

We now compute the elliptic genus $\phi_0(M)$ localized at $N_1$.  First, we fix the dimensions to be $\dim M = 4k$ and $\dim N_1 = 4l$. By the Lefschetz fixed point formula, the elliptic genus may be calculated in terms of the fixed point set $M^{\sigma_1}$. Moreover, we recall that in this formula the Euler class $e(\nu) \in H^{4k-4l}(N_1; \Q)$ of the normal bundle factors out. We show that the Euler class $e(\nu)$ vanishes.

It follows from the Connectedness Lemma that the inclusion map $N_1 \hookrightarrow M$ is   $(8l-4k+1)$--connected. By the second part of Lemma \ref{4per_fam}, we observe that the inclusion maps $N_{i+1} \hookrightarrow N_i$ are at least $(8l-4k+1)$--connected for $i \in \{0, \ldots, s-2 \}$. Moreover, we note that $(8l-4k+1) \geq 5$ and therefore, we get that
$$b_4(M) = b_4(N_r)=0.
$$
Since $N_r$ is rationally $4$--periodic and $\dim N_r > 8l-4k$, we conclude with Poincar\'e duality
$$ b_{4k-4l}(N_1) = b_{8l-4k}(N_1) = b_{8l-4k}(N_r) = b_{4}(N_r) = 0.
$$
Therefore, the Euler class $e(\nu) \in H^{4k-4l}(N_1; \Q)$ vanishes and so does the elliptic genus. Thus, we established a contradiction.
\end{proof}

We finish this paper with the proof of Corollary \ref{corC}.

\begin{proof}[Proof of Corollary \ref{corC}] In view of the result by Amann and Kennard \cite[\textit{\mbox{Theorem A}}]{AK14a} it suffices to show that the elliptic genus vanishes, when the torus fixed point set consists of two isolated fixed points or is connected and has the rational cohomology of a sphere. We distinguish these two cases.

If the torus fixed point set consists of two isolated fixed points $p$ and $q$, then the induced representations at the tangent spaces $T_p M$ and $T_q M$ are isomorphic. This follows from the Lefschetz fixed point formula and was highlighted by Atiyah-Bott \cite[\textit{Theorem 7.15, p.476}]{AB68} in a more general setting. By using, for example, the localization formula from Atiyah-Bott, it can further be shown that $M$ is rationally zero-bordant. As a consequence, the elliptic genus vanishes.

On the other hand, if the torus fixed point set is connected, then \mbox{Theorem \ref{thmC}} implies that the elliptic genus is constant as a power series and therefore it equals the signature. However, the signature vanishes, since the fixed point set is a cohomology sphere. This concludes the proof.
\end{proof}

\bibliographystyle{plain}
\bibliography{references}

\end{document}